\documentclass[11pt,twoside,letterpaper]{article} 
\usepackage{times,fancyhdr}   
\usepackage{amsfonts}                       
\usepackage{amsthm}                      
\usepackage{amsmath}                   
\usepackage{amssymb}  
\usepackage[normalem]{ulem}
\usepackage{color}

\numberwithin{equation}{section}


\makeatletter 
\def\cleardoublepage{\clearpage\if@twoside \ifodd\c@page\else%
    \hbox{}%
    \thispagestyle{empty}%
    \newpage%
    \if@twocolumn\hbox{}\newpage\fi\fi\fi} 
\makeatother 

\def \C{\mathbb{C}}

\def \Q{\mathbb{Q}}

\def \Z{\mathbb{Z}}

\newtheorem{thm}{Theorem}[section]

\newtheorem{lemma}{Lemma}[section]

\newtheorem{example}{Example}[section]

\begin{document}
\title{
{\begin{flushleft}
\vskip 0.45in
{\normalsize\bfseries\textit{ }}
\end{flushleft}
\vskip 0.45in
\bfseries\scshape Algebraic independence results for values of Jacobi theta-constants}}

\thispagestyle{fancy}
\fancyhead{}
\fancyhead[L]{In: Book Title \\ 
Editor: Editor Name, pp. {\thepage-\pageref{lastpage-01}}} 
\fancyhead[R]{ISBN 0000000000  \\
\copyright~2007 Nova Science Publishers, Inc.}
\fancyfoot{}
\renewcommand{\headrulewidth}{0pt}

\author{\bfseries\itshape Carsten Elsner\thanks{Fachhochschule f{\"u}r die Wirtschaft, University of Applied Sciences, Freundallee 15,
D-30173 Hannover, Germany \newline e-mail: carsten.elsner@fhdw.de},
Yohei Tachiya\thanks{Hirosaki University, Graduate School of Science and Technology, 
Hirosaki 036-8561, Japan \newline e-mail: tachiya@hirosaki-u.ac.jp}}

\date{}
\maketitle
\thispagestyle{empty}
\setcounter{page}{1}
\[\]
\[\]
\begin{abstract}
Let $\theta_3(\tau)=1+2\sum_{\nu=1}^{\infty} q^{\nu^2}$ with $q=e^{i\pi \tau}$ and $\Im (\tau)>0$ denote the Thetanullwert of the Jacobi
theta function 
\[\theta(z|\tau) \,=\,\sum_{\nu=-\infty}^{\infty} e^{\pi i\nu^2\tau + 2\pi i\nu z} \,.\]
Moreover, let $\theta_2(\tau)=2\sum_{\nu=0}^{\infty} q^{{(\nu+1/2)}^2}$ and $\theta_4(\tau)=1+2\sum_{\nu=1}^{\infty} {(-1)}^{\nu}q^{\nu^2}$. \\
For every even integer $n\geq 6$, which is not a power of two, we prove constructively the existence of a nontrivial integer polynomial $Q_n(X,Y)$ such that 
\[Q_n\Big( \,\frac{\theta_3^4(n\tau)}{\theta_3^4(\tau)},\frac{\theta_2^4(\tau)}{\theta_3^4(\tau)}\, \Big) \,=\, 0 \]
holds for all complex numbers $\tau$ from the upper half plane of ${\C}$. These polynomials are used to prove the algebraic independence of $\theta_3(n\tau)$
and $\theta_3(\tau)$ for all algebraic numbers $q=e^{i\pi \tau}$ with $0<|q|<1$. Combining this with former results of the authors, it is shown that for such 
algebraic $q$ the numbers $\theta_3(n\tau)$ and $\theta_3(\tau)$ are algebraically independent over ${\Q}$ for every integer $n\geq 2$. 
A result on the algebraic dependence over $\mathbb{Q}$ of the three numbers $\theta_3(\ell\tau)$, $\theta_3(m\tau)$, and $\theta_3(n\tau)$ for integers 
$\ell,m,n\geq 1$ is also presented.
\end{abstract}
 
\vspace{.08in} \noindent \textbf{Keywords:} Algebraic independence, Theta-constants, Nesterenko's theorem, 
Modular equations 


\pagestyle{fancy}  
\fancyhead{}
\fancyhead[EC]{Carsten Elsner and Yohei Tachiya}
\fancyhead[EL,OR]{\thepage}
\fancyhead[OC]{Algebraic independence results for values of Jacobi theta-constants}
\fancyfoot{}
\renewcommand\headrulewidth{0.5pt} 

\newpage
\section{Introduction and statement of the results} \label{Sec1}   
Let \(\tau \) be a complex variable in the complex upper 
half-plane $\mathbb{H}:=
\{
\tau\in\mathbb{C}\,|\,\Im(\tau)>0
\}$.
The series
\[\theta_2(\tau) = 2\sum_{\nu=0}^{\infty} q^{{(\nu+1/2)}^2} \,,\qquad \theta_3(\tau) = 1+2\sum_{\nu=1}^{\infty} 
q^{\nu^2} \,,\qquad \theta_4(\tau) = 1+2\sum_{\nu=1}^{\infty} {(-1)}^\nu q^{\nu^2} \]
are known as theta-constants or Thetanullwerte, where \(q=e^{\pi i\tau} \). In particular,  $\theta_3(\tau)$ 
is the Thetanullwert of the Jacobi theta function $\theta(z|\tau) \,=\,\sum_{\nu=-\infty}^{\infty} e^{\pi i\nu^2\tau + 2\pi i\nu z}$. 
Recently, the first-named author and his coauthor Yohei Tachiya has proven the following results.
\[\]
{\bf Theorem~A.} \cite[Theorem~1.1]{Elsner3}
{\em Let $\tau\in\mathbb{H}$ be a complex number such that 
\(q=e^{\pi i \tau} \) is an algebraic number.
Let  $m\geq1$ be an integer. Then, the two numbers $\theta_3(2^m\tau)$ and
$\theta_3(\tau)$ are algebraically independent over ${\Q}$ as well as the two numbers $\theta_3(n\tau)$ and $\theta_3(\tau)$ for $n=3,5,6,7,9,10,11,12$\/}.\\
\[\]
{\bf Theorem~B.} \cite[Theorem~1.1]{Elsner4}
{\em Let $\tau\in\mathbb{H}$ be a complex number such that 
\(q=e^{\pi i \tau} \) is an algebraic number.
Let $n\geq3$ be an odd integer. Then, the numbers in each of the sets\/} 
\[\{\theta_2(n\tau),\theta_2(\tau)\},\quad\{\theta_3(n\tau),\theta_3(\tau)\},\quad \{\theta_4(n\tau),\theta_4(\tau)\} \]
{\em are algebraically independent over ${\Q}$.\/}
\[\]
The first basic tool in proving such algebraic independence results are integer polynomials in two variables $X,Y$, which vanish at certain points $X=X_0$ and $Y=Y_0$
given by values of rational functions of theta-constants. Firstly, for Theorem~A, we used the following Theorem.
\[\]
{\bf Theorem~C.} \cite[Lemma~3.1]{Elsner3}
{\em For every integer \(\alpha \geq 1 \) let $n=2^{\alpha}$. There exists a polynomial \(P_n(X,Y) \in {\Z}[X,Y] \) such that\/}
\[P_n\Big( \frac{\theta_3^2(n\tau)}{\theta_3^2(\tau)}, \frac{\theta_4(\tau)}{\theta_3(\tau)} \Big) \,=\, 0 \] 
holds for any $\tau\in\mathbb{H}$, 
where $\deg_X P_2(X,Y)=1$, and $\deg_X P_n(X,Y)=2^{\alpha-2}$ for $\alpha \geq 2$.\/
\[\]
For instance,
\begin{eqnarray*}
P_2 &=& 2X-Y^2-1 \,,\\ 
P_4 &=& 4X-{(1+Y)}^2 \,,\\
P_8 &=& 64X^2 - 16{(1+Y)}^2X + {(1-Y)}^4 \,,\\
P_{16} &=& 65536X^4 - 16384{(1+Y)}^2X^3 + 512(3Y^4+4Y^3+18Y^2+4Y+3)X^2 \\
&& -\,64{(1+Y)}^2(Y^4+28Y^3+6Y^2+28Y+1)X + {(1-Y)}^8 \,.
\end{eqnarray*}     
In order to state the identities underlying Theorem~B we need some auxiliary functions.
Let \(n\geq 3 \) denote an odd positive integer. Set
\[h_j(\tau) := n^2\frac{\theta_j^4(n\tau)}{\theta_j^4(\tau)} \quad (j=2,3,4)\,,\quad 
\lambda = \lambda(\tau) := \frac{\theta_2^4(\tau)}{\theta_3^4(\tau)} \,,
\quad \psi (n) := n\prod_{p|n} \Big( 1+\frac{1}{p} \Big) \,,\]
where \(p\) runs through all primes dividing \(n\). Yu.V. Nesterenko \cite{Nes3} proved the existence of integer polynomials $P_n(X,Y)\in {\Z}[X,Y]$ such that
$P_n\big(h_j(\tau),R_j(\lambda(\tau))\big)=0$ holds for $j=2,3,4$, odd integers $n\geq 3$, and a suitable rational function $R_2,R_3$, or $R_4$, respectively.
\[\]
{\bf Theorem~D.} \cite[Theorem~1, Corollaries~3-4]{Nes3} 
{\em For any odd integer \(n\geq 3 \) there exists a polynomial \(P_n(X,Y) \in {\Z}[X,Y] \), \(\deg_X P_n(X,Y) = \psi (n) \), \(\deg_Y P_n \leq (n-1)\psi(n)/n \), 
such that\/}
\begin{eqnarray*}
P_n\Big( \,h_2(\tau),16\frac{\lambda(\tau)-1}{\lambda(\tau)} \,\Big) &\,=\,& 0 \,,\\
P_n\big( h_3(\tau),16\lambda(\tau) \big) &\,=\,& 0 \,,\\
P_n\Big( \,h_4(\tau),16\frac{\lambda(\tau)}{\lambda(\tau)-1} \,\Big) &\,=\,& 0 \,
\end{eqnarray*}
hold for any $\tau\in\mathbb{H}$. 
\[\]

For instance,
\begin{eqnarray*}
P_3 &=& 9-(28-16Y+Y^2)X+30X^2-12X^3+X^4 \,,\\
P_5 &=& 25-(126-832Y+308Y^2-32Y^3+Y^4)X+(255+1920Y-120Y^2)X^2 \\
&& +\,(-260+320Y-20Y^2)X^3+135X^4-30X^5+X^6 \,.
\end{eqnarray*}
The polynomials $P_7,P_9$, and $P_{11}$ are listed in the appendix of \cite{Elsner3}. $P_3$ and $P_5$ are already given in \cite{Nes3},
$P_7,P_9$, and $P_{11}$ are the results of computer-assisted computations of the first-named author. \\
\\
In this paper we focus on the problem to fill the gap between Theorem~A and Theorem~B by considering all positive even integers which are not a power of two.
Moreover, we restrict our investigations on the function $\theta_3(\tau)$. 
In the following theorems, let $\tau\in\mathbb{H}$ be a complex number such that 
\(q=e^{\pi i \tau} \) is an algebraic number.
\begin{thm}
Let \(n\geq 6 \) be an even integer which is not a power of two. Then the numbers $\theta_3(n\tau)$ and $\theta_3(\tau)$ are algebraically independent over ${\Q}$.
\label{Thm1} 
\end{thm}
Collecting together the algebraic independence results for the function $\theta_3$ from Theorem~A, Theorem~B, and Theorem~\ref{Thm1}, we obtain the following main
theorem.
\begin{thm}
Let \(n\geq 2 \) be an integer. Then the numbers $\theta_3(n\tau)$ and $\theta_3(\tau)$ are algebraically independent over ${\Q}$.
\label{Thm2} 
\end{thm}
\begin{thm}
Let $\ell,m,n\geq1$ be integers. Then the three numbers 
$\theta_3(\ell\tau),\theta_3(m\tau)$, and $\theta_3(n\tau)$ are algebraically dependent over $\mathbb{Q}$.
\label{Thm3}
\end{thm}
Theorem~\ref{Thm1} is founded again on the existence of certain integer polynomials in two variables $X$ and $Y$, which vanish for 
$X=\theta_3^4(n\tau)/\theta_3^4(\tau)$ and $Y=\theta_2^4(\tau)/\theta_3^4(\tau)$ (Lemma~\ref{lem:3}). In order to handle the specific properties of these polynomials,
it is necessary to compute $P_m(0,Y)$ for the polynomials $P_m(X,Y)$ from Theorem~D, where $m\geq 3$ denotes any odd integer. It turns out that $P_m(0,Y)$ is a
nonvanishing constant polynomial. In Section~\ref{Sec2} we prove this fact by expressing the finite product of certain theta functions by infinite products of
polynomials in $q$. Here, Jacobi's triple product formula plays an important role. Theorem~\ref{Thm3} can be proven by Theorem~\ref{Thm2} and considering
the transcendence degrees of certain field extensions.

\section{The formula of $\theta_3(\tau)$}\label{Sec2}

\begin{thm}\label{thm:1}
Let $n\geq 3$ be an odd integer, and let the number $\psi(n)$ be defined as in Nesterenko's paper \cite{Nes3}. Then we have
\begin{equation}\label{eq:123}
\prod \theta_3\left(
\frac{u\tau +2v}{w}\right) \,=\,\theta_3^{\psi(n)}(\tau),
\end{equation}
where the product is taken for all $\psi(n)$ triplets $u,v,w$ of nonnegative integers satisfying the conditions
\begin{equation}\label{eq:01}
(u,v,w)\,=\, 1 \,,\qquad uw \,=\, n \,,\qquad 0 \,\leq \,v\,<\,w \,.
\end{equation}
\end{thm}  
   
For the sake of brevity, we put $\theta_3(\tau)=\theta_3(q)$ 
$(q=e^{i\pi \tau},\tau\in\mathbb{H})$.
Then the equality (\ref{eq:123}) is equivalent to 
\begin{equation}\label{eq:081}
\prod_{u,w}^{}\prod_{v=0\atop (u,v,w)=1}^{w-1}
\theta_3(\zeta_w^vq^{u/w})=\theta_3^{\psi(n)}(q),
\end{equation}
where $\zeta_w:=e^{2\pi i/w}$ be a primitive 
$w$th root of unity and the product is taken for $u$ and $w$ with $uw=n$. 
In what follows, we prepare some lemmas for the proof of the equality (\ref{eq:081}). 
\begin{lemma}\label{lem:01}
Let $p$ be an odd prime. Then 
\begin{equation}\label{eq:0}
\prod_{k=0}^{p-1}\theta_3(\zeta_p^kq)=\frac{\theta_3^{p+1}(q^p)}{\theta_3(q^{p^2})}.
\end{equation}
\end{lemma}

\begin{proof}
Define 
$$
F(q):=\prod_{\ell=1}^{\infty}(1-q^\ell)
$$
for $q\in\mathbb{C}$ with $|q|<1$. 
Then we have 
$$
\begin{array}{ll}
\displaystyle\prod_{k=0}^{p-1}F(\zeta_p^kq)&=
\displaystyle\prod_{\ell=1}^{\infty}
\prod_{k=0}^{p-1}
(1-(\zeta_p^kq)^\ell)\\\\
{}&=\left(\displaystyle\prod_{p\mid \ell}
\prod_{k=0}^{p-1}
(1-(\zeta_p^kq)^\ell)
\right)
\left(\displaystyle\prod_{p\mid \hspace{-.27em}/ \ell}
\prod_{k=0}^{p-1}
(1-(\zeta_p^kq)^\ell)
\right)\\\\
{}&=
\displaystyle\prod_{p\mid \ell}
(1-q^{\ell})^p
\displaystyle\prod_{p\mid \hspace{-.27em}/ \ell}
(1-q^{p\ell})
\\\\
{}&=
\displaystyle\prod_{\ell=1}^{\infty}
(1-q^{p\ell})^p
\displaystyle\prod_{p\mid \hspace{-.27em}/ \ell}
(1-q^{p\ell}),
\end{array}
$$
and hence 
\begin{eqnarray}
F(q^{p^2})\displaystyle\prod_{k=0}^{p-1}F(\zeta_p^kq)&=&
\displaystyle\prod_{\ell=1}^{\infty}(1-q^{p^2\ell})\displaystyle\prod_{\ell=1}^{\infty}
(1-q^{p\ell})^p
\displaystyle\prod_{p\mid \hspace{-.27em}/ \ell}
(1-q^{p\ell})\nonumber \\ \nonumber \\ \nonumber
{}&=&\displaystyle\prod_{\ell=1}^{\infty}
(1-q^{p\ell})^{p+1}\label{eq:29}\\ \nonumber \\ 
{}&=&F(q^p)^{p+1}.
\end{eqnarray}

On the other hand, using Jacobi's triple product expression 
for $\theta_3(q)$, we have 
\begin{eqnarray}
\theta_3(q)&
=&\displaystyle\prod_{\ell=1}^{\infty}(1-q^{2\ell})(1+q^{2\ell-1})^2\nonumber\\\nonumber\\
{}&=&\displaystyle\prod_{\ell=1}^{\infty}\frac{(1-q^{2\ell})^5}{(1-q^\ell)^2(1-q^{4\ell})^2}\nonumber\\\nonumber\\
{}&=&\displaystyle\frac{F(q^2)^5}{F(q)^2F(q^4)^2},\label{eq:30}
\end{eqnarray}
where we used the equalities
$$
\prod_{\ell=1}^{\infty}(1+q^{2\ell-1})=\prod_{\ell=1}^{\infty}
\frac{1+q^\ell}{1+q^{2\ell}}=\prod_{\ell=1}^{\infty}
\frac{(1-q^{2\ell})^2}{(1-q^\ell)(1-q^{4\ell})}.
$$
Therefore we obtain by (\ref{eq:29}) and (\ref{eq:30})
$$
\begin{array}{ll}
\theta_3(q^{p^2})
\displaystyle\prod_{k=0}^{p-1}\theta_3(\zeta_p^kq)
&=\displaystyle\frac{F(q^{2p^2})^5}{F(q^{p^2})^2F(q^{4{p^2}})^2}
\prod_{k=0}^{p-1}\displaystyle\frac{F(\zeta_p^{2k}q^{2})^5}{F(\zeta_p^kq)^2F(\zeta_p^{4k}q^{4})^2}\\\\
{}&=\displaystyle\frac{F(q^{2p^2})^5}{F(q^{p^2})^2F(q^{4p^2})^2}
\prod_{k=0}^{p-1}\displaystyle\frac{F(\zeta_p^{k}q^{2})^5}{F(\zeta_p^kq)^2F(\zeta_p^{k}q^{4})^2}\\\\
{}&=\left(\displaystyle\frac{F(q^{2p})^5}{F(q^p)^2F(q^{4p})^2}\right)^{p+1}\\\\
{}&=\theta_3(q^p)^{p+1},
\end{array}
$$
which is our desired.
\end{proof}
\begin{lemma}\label{lem:02}
Let $p$ be an odd prime. For any integer $j\geq0$, we have 
\begin{equation}\label{eq:2}
\prod_{k=0}^{p^j-1}\theta_3(\zeta_{p^j}^kq)=
\frac{\theta_3^{b_{j+1}}(q^{p^j})}
{\theta_3^{b_j}(q^{p^{j+1}})},
\end{equation}
where 
$b_n$ $(n\geq0)$ are nonnegative integers defined by  
$$
b_n:=\frac{p^n-1}{p-1}.
$$
\end{lemma}
\begin{proof}
The assertion is trivial for $j=0$ and Lemma \ref{lem:01}
is the case of $j=1$.  
Suppose that the equality (\ref{eq:2}) holds for $j\geq1$. 
Then, for each $v=0,1,\dots,p-1$, 
we replace $q$ by $\zeta_{p^{j+1}}^v q$ in (\ref{eq:2});
$$
\prod_{k=0}^{p^j-1}\theta_3(\zeta_{p^j}^k\zeta_{p^{j+1}}^v q)=
\frac{\theta_3^{b_{j+1}}(\zeta_p^v q^{p^j})}
{\theta_3^{b_j}(q^{p^{j+1}})},
$$
where $\zeta_{p^j}^k\zeta_{p^{j+1}}^v=\zeta_{p^{j+1}}^{pk+v}$. 
Taking the product of the both sides above for $v$ from 
$v=0$ to $p-1$, we get 
\begin{equation}\label{eq:4}
\prod_{v=0}^{p-1}
\prod_{k=0}^{p^j-1}
\theta_3(\zeta_{p^{j+1}}^{pk+v}q)=
\frac{\prod_{v=0}^{p-1}\theta_3^{b_{j+1}}(\zeta_p^vq^{p^j})}
{\theta_3^{pb_j}(q^{p^{j+1}})}.
\end{equation}
Since the integers $pk+v$ $(0\leq k< p^j, 0\leq v\leq p-1)$ are distinct, 
\begin{equation}\label{eq:41}
\prod_{v=0}^{p-1}
\prod_{k=0}^{p^j-1}
\theta_3(\zeta_{p^{j+1}}^{pk+v}q)=
\prod_{k=0}^{p^{j+1}-1}\theta_3(\zeta_{p^{j+1}}^kq).
\end{equation}
Furthermore, replacing $q$ by $q^{p^j}$ in (\ref{eq:0}), we have
\begin{equation}\label{eq:43}
\prod_{v=0}^{p-1}\theta_3(\zeta_p^vq^{p^j})=
\frac{\theta_3^{p+1}(q^{p^{j+1}})}{\theta_3(q^{p^{j+2}})}.
\end{equation}
Applying the results (\ref{eq:41}) and (\ref{eq:43}) to (\ref{eq:4}), we obtain
$$
\begin{array}{ll}
\displaystyle\prod_{k=0}^{p^{j+1}-1}
\theta_3(\zeta_{p^{j+1}}^{k}q)
&=\left(
\displaystyle\frac{\theta_3^{p+1}(q^{p^{j+1}})}{\theta_3(q^{p^{j+2}})}
\right)^{b_{j+1}}
\displaystyle\frac{1}{
\theta_3^{pb_j}(q^{p^{j+1}})}\\\\
{}&
=
\displaystyle\frac{\theta_3^{(p+1)b_{j+1}-pb_j}(q^{p^{j+1}})}
{\theta_3^{b_{j+1}}(q^{p^{j+2}})}\\\\
{}&=\displaystyle\frac{\theta_3^{b_{j+2}}(q^{p^{j+1}})}
{\theta_3^{b_{j+1}}(q^{p^{j+2}})}.
\end{array}
$$
This implies that the equality (\ref{eq:2}) also holds for $j+1$, 
and hence Lemma~\ref{lem:02} is proved. 
\end{proof}
\begin{lemma}\label{lem:03}
The equality (\ref{eq:081}) holds for any power of odd prime $n=p^\ell$ $(\ell\geq1)$, namely, 
\begin{equation}\label{eq:51}
\prod_{j=0}^{\ell} 
\prod_{v=0\atop (p^{\ell-j},v,p^j)=1}^{p^j-1}
\theta_3
(\zeta_{p^j}^vq^{p^{\ell-2j}})
=\theta_3^{\psi(p^\ell)}(q).
\end{equation}
\end{lemma}
\begin{proof} 
If $n=p^\ell$ $(\ell\geq1)$ is a power of odd prime, 
then the product at the left hand side in (\ref{eq:081}) is 
taken for all triples $u=p^{\ell-j}$, $v$, $w=p^j$ satisfying the conditions
$$
(p^{\ell-j},v,p^j)=1, \qquad 0\leq v<p^j,\qquad 0\leq j\leq \ell,
$$ 
and hence the product is given by the left hand side in (\ref{eq:51}).

In what follows, we show the equality (\ref{eq:51}). 
Let $j\geq0$ be a fixed integer. 
We first replace $q$ by $q^{p^{\ell-2j}}$ in (\ref{eq:2}); 
\begin{equation}\label{eq:61}
\prod_{v=0}^{p^j-1}
\theta_3(\zeta_{p^j}^vq^{p^{\ell-2j}})=
\frac
{\theta_3^{b_{j+1}}
(
(q^{p^{\ell-2j}})^{p^j}
)}
{\theta_3^{b_j}
(
(q^{p^{\ell-2j}})^{p^{j+1}}
)}=
\frac{\theta_3^{b_{j+1}}(q^{p^{\ell-j}})}
{\theta_3^{b_j}(q^{p^{\ell-j+1}})},
\end{equation}
and replace $j$ by $j-1$ and then $q$ by $q^{p^{\ell-2j}}$ in (\ref{eq:2});
\begin{equation}\label{eq:62}
\prod_{v=0}^{p^{j-1}-1}
\theta_3(\zeta_{p^{j-1}}^vq^{p^{\ell-2j}})=
\frac
{\theta_3^{b_{j}}
(
(q^{p^{\ell-2j}})^{p^{j-1}}
)}
{\theta_3^{b_{j-1}}
(
(q^{p^{\ell-2j}})^{p^{j}}
)}=
\frac{\theta_3^{b_{j}}(q^{p^{\ell-j-1}})}
{\theta_3^{b_{j-1}}(q^{p^{\ell-j}})}.
\end{equation}
Since the condition
$
(p^{\ell-j},v,p^j)=1$ is equivalent to that 
$j=0,1$ or $p\mid \hspace{-.67em}/ v$ 
when $j\neq0,1$, 
we have by (\ref{eq:61}) and (\ref{eq:62})
$$
\begin{array}{ll}
\displaystyle\prod_{j=0}^{\ell} 
\prod_{v=0\atop (p^{\ell-j},v,p^j)=1}^{p^j-1}
\theta_3
(\zeta_{p^j}^vq^{p^{\ell-2j}})&=
\theta_3(q^{p^\ell})
\left(
\displaystyle\prod_{j=1}^{\ell-1}
\displaystyle\prod_{v=0\atop p\mid \hspace{-.27em}/ v}^{p^j-1}
\theta_3(\zeta_{p^j}^vq^{p^{\ell-2j}})
\right)
\displaystyle\prod_{v=0}^{p^\ell-1}
\theta_3(\zeta_{p^{\ell}}^vq^{p^{-\ell}})\\\\
{}&=
\theta_3(q^{p^\ell})
\left(
\displaystyle\prod_{j=1}^{\ell-1}
\displaystyle\frac{\prod_{v=0}^{p^j-1}
\theta_3(\zeta_{p^j}^vq^{p^{\ell-2j}})}
{\prod_{v=0}^{p^{j-1}-1}
\theta_3(\zeta_{p^{j-1}}^{v}q^{p^{\ell-2j}})}
\right)
\displaystyle\prod_{v=0}^{p^\ell-1}
\theta_3(\zeta_{p^\ell}^vq^{p^{-\ell}})\\\\
{}&=\displaystyle
\frac{
\prod_{j=0}^{\ell}
\prod_{v=0}^{p^j-1}
\theta_3(\zeta_{p^j}^vq^{p^{\ell-2j}})}
{\prod_{j=1}^{\ell-1}\prod_{v=0}^{p^{j-1}-1}
\theta_3(\zeta_{p^{j-1}}^{v}q^{p^{\ell-2j}})
}
\\\\
{}&=
\left(
\displaystyle\prod_{j=0}^{\ell}
\frac{\theta_3^{b_{j+1}}(q^{p^{\ell-j}})}
{\theta_3^{b_{j}}(q^{p^{\ell-j+1}})}
\right)
\left(
\displaystyle\prod_{j=1}^{\ell-1}
\frac{\theta_3^{b_{j-1}}(q^{p^{\ell-j}})}{\theta_3^{b_{j}}(q^{p^{\ell-j-1}})}
\right)
\\\\
&=
\theta_3^{b_{\ell+1}-b_{\ell-1}}(q)\\\\
{}&
=\theta_3^{\psi(p^\ell)}(q),
\end{array}
$$
where we used the following equalities at the second equality above;
\begin{equation}\label{eq:01234}
\displaystyle\prod_{v=0\atop p\mid \hspace{-.27em}/ v}^{p^j-1}
\theta_3(\zeta_{p^j}^vq^{\ell-2j})
=
\displaystyle\frac{\prod_{v=0}^{p^j-1}
\theta_3(\zeta_{p^j}^vq^{\ell-2j})}
{\prod_{v=0}^{p^{j-1}-1}
\theta_3(\zeta_{p^j}^{pv}q^{\ell-2j})}\\\\
=\displaystyle\frac{\prod_{v=0}^{p^j-1}
\theta_3(\zeta_{p^j}^vq^{\ell-2j})}
{\prod_{v=0}^{p^{j-1}-1}
\theta_3(\zeta_{p^{j-1}}^{v}q^{\ell-2j})}.
\end{equation}
Thus the proof of Lemma \ref{lem:3} is completed.
\end{proof}

\noindent
{\it Proof of Theorem \ref{thm:1}}. 
We prove the equality (\ref{eq:081}) by the induction 
on the number of distinct prime factors of an odd integer $n$. 
If the integer $n$ has only one prime factor, namely, 
$n$ is a power of some odd prime, then the assertion follows immediately from 
Lemma~\ref{lem:03}.  
Suppose that the equality (\ref{eq:081}) holds for the odd integer $m\geq3$ 
having the $s(\geq1)$ distinct prime factors;
\begin{equation}\label{eq:091}
\prod_{u,w}^{}\prod_{v=0\atop (u,v,w)=1}^{w-1}
\theta_3(\zeta_w^vq^{u/w})=\theta_3^{\psi(m)}(q),
\end{equation}
where the product is taken for $u$ and $w$ with $uw=m$. 

Let $\ell\geq1$ be an integer and $p$ be an odd prime number not dividing $m$. 
 In what follows, under the induction hypothesis (\ref{eq:091}), we prove that the equality (\ref{eq:081}) also holds for the integer 
$n=mp^{\ell}$ having the $s+1$ distinct prime factors. 
If $n=mp^\ell$ $(\ell\geq1)$, then the product at the left hand side in (\ref{eq:081}) is 
taken for all triples $u'=up^{\ell-j}$, $v'=v$, $w'=wp^j$ satisfying the conditions
$$
(u',v',w')=(up^{\ell-j},v,wp^j)=1,\qquad uw=m,\qquad 0\leq v<wp^j,\qquad 0\leq j\leq \ell,
$$ 
and hence the product 
is given by 
\begin{equation}\label{eq:991}
\prod_{j=0}^{\ell}
\prod_{u,w}
\prod_{v=0\atop (up^{\ell-j},v,wp^{j})=1}^{wp^j-1} \theta_3(
\zeta_{wp^j}^vq^{\frac{u}{w}{p^{\ell-2j}}}),
\end{equation}
where the product is taken for $u$ and $w$ with $uw=m$. 
Since $m=uw$ is not divided by $p$, we see that the equivalent conditions
$$(up^{\ell-j},v,wp^j)=1\Longleftrightarrow
\left\{
\begin{array}{ll}
(u,v,w)=1,&{\rm if}\quad j=0,\ell,\\
(u,v,w)=1\quad{\rm and}\quad p\mid \hspace{-.67em}/ v,&{\rm otherwise},\\
\end{array}
\right.
$$ 
and hence the product (\ref{eq:991}) is divided into the three parts as follows;
$$
\left(
\prod_{u,w}\prod_{v=0\atop (u,v,w)=1}^{w-1}
\theta_3(\zeta_w^vq^{\frac{u}{w}p^\ell})
\right)
\left(
\prod_{j=1}^{\ell-1}
\prod_{u,w}
\prod_{v=0\atop{(u,v,w)=1 \atop{p\mid \hspace{-.27em}/ v}}}^{p^jw-1}
\theta_3(\zeta_{p^jw}^vq^{\frac{u}{w}p^{\ell-2j}})
\right)
\left(
\prod_{u,w}\prod_{v=0\atop (u,v,w)=1}^{p^\ell w-1}\theta_3(\zeta_{wp^\ell}^vq^{\frac{u}{w}p^{-\ell}})
\right),
$$
where, similarly as in (\ref{eq:01234}), we have  
$$
\prod_{v=0\atop{(u,v,w)=1 \atop{p\mid \hspace{-.27em}/ v}}}^{p^jw-1}
\theta_3(\zeta_{p^jw}^vq^{\frac{u}{w}p^{\ell-2j}})
=
\displaystyle
\frac{\displaystyle\prod_{v=0\atop{(u,v,w)=1}}^{p^jw-1}
\theta_3(\zeta_{p^jw}^vq^{\frac{u}{w}p^{\ell-2j}})}{\displaystyle\prod_{v=0\atop{(u,v,w)=1}}^{p^{j-1}w-1}
\theta_3(\zeta_{p^{j-1}w}^vq^{\frac{u}{w}p^{\ell-2j}})}.
$$
Therefore the product (\ref{eq:991}) is rewritten again as 
\begin{equation}\label{eq:1002}
\prod_{j=0}^{\ell}
\prod_{u,w}
\prod_{v=0\atop (up^{\ell-j},v,wp^{j})=1}^{wp^j-1} \theta_3(
\zeta_{wp^j}^vq^{\frac{u}{w}{p^{\ell-2j}}})=
\frac{
\displaystyle\prod_{j=0}^{\ell}\prod_{u,w}\prod_{v=0\atop{(u,v,w)=1}}^{p^jw-1}
\theta_3(\zeta_{p^jw}^vq^{\frac{u}{w}p^{\ell-2j}})}
{\displaystyle\prod_{j=1}^{\ell-1}\prod_{u,w}\prod_{v=0\atop{(u,v,w)=1}}^{p^{j-1}w-1}
\theta_3(\zeta_{p^{j-1}w}^vq^{\frac{u}{w}p^{\ell-2j}})}.
\end{equation}

Now we simplify the products at the right hand side above.  
Let $j\geq0$ be a fixed integer. For each $k=0,1,\dots,p^j-1$, we replace $q$ by $\zeta_{p^j}^kq^{p^{\ell-2j}}$ in (\ref{eq:091});
\begin{equation}\label{eq:095}
\prod_{u,w}\prod_{v=0\atop (u,v,w)=1}^{w-1}
\theta_3(\zeta_{p^jw}^{p^jv+uk}q^{\frac{u}{w}p^{\ell-2j}})=
\theta_3^{\psi(m)}(\zeta_{p^j}^kq^{p^{\ell-2j}}).
\end{equation}
Taking the product of the both sides above for $k$ from $0$ to $p^j-1$;
\begin{equation}\label{eq:7}
\prod_{k=0}^{p^j-1}
\prod_{u,w}\prod_{v=0\atop (u,v,w)=1}^{w-1}
\theta_3(\zeta_{p^jw}^{p^jv+uk}q^{\frac{u}{w}p^{\ell-2j}})=
\prod_{k=0}^{p^j-1}
\theta_3^{\psi(m)}(\zeta_{p^j}^kq^{p^{\ell-2j}}).
\end{equation}
Then the left hand side in (\ref{eq:7}) is expressed by   
\begin{equation}\label{eq:8}
\prod_{k=0}^{p^j-1}
\prod_{u,w}\prod_{v=0\atop (u,v,w)=1}^{w-1}
\theta_3(\zeta_{p^jw}^{p^jv+uk}q^{\frac{u}{w}p^{\ell-2j}})
=
\prod_{u,w}\prod_{v=0\atop{(u,v,w)=1}}^{p^jw-1}
\theta_3(\zeta_{p^jw}^vq^{\frac{u}{w}p^{\ell-2j}}).
\end{equation}
On the other hand, by (\ref{eq:61}), the right hand side in (\ref{eq:7}) is expressed by 
\begin{equation}\label{eq:9}
\displaystyle\prod_{k=0}^{p^j-1}\theta_3^{\psi(m)}(\zeta_{p^j}^kq^{p^{\ell-2j}})=
\left(
\displaystyle\frac{\theta_3^{b_{j+1}}(q^{p^{\ell-j}})}
{\theta_3^{b_j}(q^{p^{\ell-j+1}})}
\right)^{\psi(m)}.
\end{equation}
Thus, by (\ref{eq:8}) and (\ref{eq:9}),  we can rewrite the equality (\ref{eq:7}) as follows;
$$
\prod_{u,w}\prod_{v=0\atop{(u,v,w)=1}}^{p^jw-1}
\theta_3(\zeta_{p^jw}^vq^{\frac{u}{w}p^{\ell-2j}})=
\left(
\displaystyle\frac{\theta_3^{b_{j+1}}(q^{p^{\ell-j}})}
{\theta_3^{b_j}(q^{p^{\ell-j+1}})}
\right)^{\psi(m)}.
$$ 
Taking the product of both hand sides for $j$ from $0$ to $\ell$, we have
\begin{equation}\label{eq:10}
\prod_{j=0}^{\ell}\left(
\prod_{u,w}\prod_{v=0\atop{(u,v,w)=1}}^{p^jw-1}
\theta_3(\zeta_{p^jw}^vq^{\frac{u}{w}p^{\ell-2j}})
\right)
=
\theta_3^{b_{\ell+1}\psi(m)}(q).
\end{equation}
Similarly, by replacing $q$ by 
$\zeta_{p^{j-1}}^kq^{p^{\ell-2j}}$ in (\ref{eq:091}) and 
proceeding the same argument above, we get
\begin{equation}\label{eq:12}
\prod_{j=1}^{\ell-1}\left(
\prod_{u,w}\prod_{v=0\atop{(u,v,w)=1}}^{p^jw-1}
\theta_3(\zeta_{p^jw}^vq^{\frac{u}{w}p^{\ell-2j}})
\right)
=
\theta_3^{b_{\ell-1}\psi(m)}(q).
\end{equation}
Therefore, applying the consequences (\ref{eq:10}) and (\ref{eq:12}) to  (\ref{eq:1002}), we have 
$$
\begin{array}{ll}
\displaystyle\prod_{u,w}
\prod_{j=0}^{\ell}\prod_{v=0\atop (up^{\ell-j},v,wp^{j})=1}^{wp^j-1} \theta_3(
\zeta_{wp^j}^vq^{\frac{u}{w}{p^{\ell-2j}}})&
=
\theta_3^{(b_{\ell+1}-b_{\ell-1})\psi(m)}(q)\\\\
{}&=\theta_3^{\psi(n)}(q).
\end{array}
$$
The proof of Theorem \ref{thm:1} is completed. \qed

\section{Lemmas}\label{Sec3}
By $\deg Q$ we denote the total degree of an integer polynomial $Q$ in one or two variables. Note that in particular $\deg 0 = -\infty$.
\begin{lemma}\label{lem:1}
Let $m\geq3$ be an odd integer. 
Then there is a polynomial $Q_m(X,Y)\in {\Z}[X,Y]$ such that 
\begin{equation}\label{eq:097}
Q_m\Big( \,\frac{\theta_3^4(m\tau)}{\theta_3^4},
\frac{\theta_2^4}{\theta_3^4}\,\Big) \,=\, 0 \,
\end{equation}
holds for any $\tau\in\mathbb{H}$, where $\deg Q_m(X,Y)=\psi(m)$ and $\deg Q_m(0,Y)=0$.
\end{lemma}
\begin{proof}
The assertion (\ref{eq:097}) follows immediately from \cite[Theorem~1]{Nes3}. For the statement on the total degree of $Q_m(X,Y)$, we have by 
\cite[Corollary~4]{Nes3}
\begin{equation}\label{eq:632}
Q_m(X,Y) \,=\, m^{2\psi(m)}X^{\psi(m)} + \sum_{\nu =1}^{\psi(m)} R_{\nu}(Y)X^{\psi(m)-\nu}\,,
\end{equation}
where for each $\nu$ with $1\leq \nu \leq \psi(m)$ 
\[\deg \big( R_{\nu}(Y)X^{\psi(m) - \nu}\big) \,\leq \, \nu \cdot \Big( \,1 - \frac{1}{m}\,\Big) + \big(\psi(m) - \nu \big) \,<\, \psi(m) \,.\] 
Thus, the total degree of $Q_m(X,Y)$ equals to $\psi(m)$. \\
Finally we prove that $Q_m(0,Y)$ is a nonzero constant. 
In Nesterenko's paper \cite[pp.154, lines 21--24]{Nes3}, 
the polynomial $R_{\psi(m)}(Y)$ in (\ref{eq:632}) is chosen such that the function in $\tau$, defined by 
\begin{equation}\label{eq:582}
\displaystyle\prod u^2\frac{\theta_3^4(\frac{u\tau +2v}{w})}{\theta_3^4(\tau)} - R_{\psi(m)}(16\lambda(\tau)) \,, 
\end{equation}
is identically zero, where the product is taken for all $\psi(m)$ triplets $u,v,w$ of nonnegative integers satisfying the conditions in (\ref{eq:01}). 
Hence, by Theorem \ref{thm:1}, the polynomial $R_{\psi(m)}(Y)$ is given by the constant
\begin{equation}\label{eq:comment}
R_{\psi(m)}(Y) \,=\, \prod u^2 \frac{\theta_3^4(\frac{u\tau +2v}{w})}{\theta_3^4(\tau)} \,=\, \prod u^2 \,.
\end{equation}
Then by (\ref{eq:632}) we have $Q_m(0,Y)=R_{\psi(m)}(Y)$, and hence the lemma is proved.
 \end{proof}
\begin{example}\label{Ex1}
We know from Theorem~D in Section~\ref{Sec1} that
\[Q_3(X,Y) \,=\, 9-(252-2304Y+2304Y^2)X+2430X^2-8748X^3+6561X^4 \]
with $\deg Q_3(X,Y)=4$ and $Q_3(0,Y)=9$.  

\end{example} 
 
\begin{lemma}\label{lem:3}
Let $n=2^{\alpha}m$ be an integer with $\alpha \geq 1$ and an odd integer $m\geq 3$. 
Then there exists a polynomial $Q_n(X,Y)\in\mathbb{Z}[X,Y]$ such that 
\begin{equation} \label{eq:-1}
Q_n\left(\displaystyle\frac{\theta_3^4(n\tau)}{\theta_3^4},\frac{\theta_2^4}{\theta_3^4}\right)=0
\end{equation}
for any complex number $\tau$ with $\Im(\tau)>0$. Furthermore, the 
polynomial $Q_n(X,Y)$ is of the form 
\begin{equation}\label{eq:leadcoeff} 
Q_n(X,Y) \,=\, c^{2^{\alpha}}Y^{2^\alpha\psi(m)}+
\sum_{j=0}^{2^\alpha\psi(m)-1}R_{n,j}(X)Y^j   
\end{equation}
with
\begin{equation}\label{eq:leadcoeff2}
Q_n(0,Y) \,=\, c^{2^{\alpha}}Y^{2^\alpha\psi(m)}\,,
\end{equation}
where $\deg R_{n,j}(X)\leq2^\alpha\psi(m)-j$ $(0\leq j< 2^{\alpha}\psi(m))$, and $c$ equals to the nonzero integer $P_m(0,Y)$, which exists by Lemma~\ref{lem:1}.
\end{lemma}
\begin{proof}
Throughout this proof any capital character with subscript(s) defines an integer polynomial.
We prove the lemma by induction with respect to $\alpha$. 
First we treat the case $\alpha =1$. Let $m\geq 3$ be an odd integer and  
\begin{equation}\label{eq:P}
Q_m(X,Y)=\sum_{\nu,\mu}a_{\nu,\mu}X^{\nu}Y^{\mu}
\end{equation}
be as in Lemma~\ref{lem:1}, where $\deg Q_m(X,Y)=\psi(m)$.
Then we have 
\begin{equation} \label{eq:35}
0=Q_m\left(\displaystyle\frac{\theta_3^4(2m\tau)}{\theta_3^4(2\tau)},\frac{\theta_2^4(2\tau)}{\theta_3^4(2\tau)}\right)
=\displaystyle\sum_{\nu,\mu}a_{\nu,\mu}\left(\frac{\theta_3^4(2m\tau)}{\theta_3^4(2\tau)}\right)^{\nu}
\left(\frac{\theta_2^4(2\tau)}{\theta_3^4(2\tau)}\right)^{\mu}
\end{equation}
for any $\tau$ with $\Im(\tau)>0$. 
Multiplying this identity with $(1+\theta_4^2/\theta_3^2)^{2\psi(m)}$, we obtain 
\begin{eqnarray}
0&=&\left(
1+\frac{\theta_4^2}{\theta_3^2}\right)^{2\psi(m)}
\displaystyle\sum_{\nu,\mu}a_{\nu,\mu}
\theta_3^{4\nu}(2m\tau)\cdot
\theta_2^{4\mu}(2\tau)\cdot
\theta_3^{-4(\nu+\mu)}(2\tau) \nonumber \\\nonumber \\
{}&=&\left(
1+\frac{\theta_4^2}{\theta_3^2}\right)^{2\psi(m)}
\displaystyle\sum_{\nu,\mu}a_{\nu,\mu}
\theta_3^{4\nu}(2m\tau)\cdot
\left(
\frac{\theta_3^2-\theta_4^2}{2}
\right)^{2\mu}\cdot
\left(
\frac{\theta_3^2+\theta_4^2}{2}
\right)^{-2(\nu+\mu)}\nonumber \\\nonumber \\
{}&=&
\displaystyle\sum_{\nu,\mu}2^{2\nu}a_{\nu,\mu}
\left(
\frac{\theta_3(2m\tau)}{\theta_3}
\right)^{4\nu}\cdot
\left(1-
\frac{\theta_4^2}{\theta_3^2}
\right)^{2\mu}\cdot
\left(1+
\frac{\theta_4^2}{\theta_3^2}
\right)^{2(\psi(m)-\nu-\mu)}\,,\label{eq:1}
\end{eqnarray}
where we used the identities 
$$
\begin{array}{ll}
2\theta_2^2(2\tau)=\theta_3^2(\tau)-\theta_4^2(\tau),\\
2\theta_3^2(2\tau)=\theta_3^2(\tau)+\theta_4^2(\tau).
\end{array}
$$
Let $n=2m$ and define 
\begin{equation} \label{eq:B}
B_n(X,Y):=\displaystyle\sum_{\nu,\mu}2^{2\nu}a_{\nu,\mu}
X^{4\nu}
\left(1-
Y^2
\right)^{2\mu}
\left(1+
Y^2
\right)^{2(\psi(m)-\nu-\mu)}\,.
\end{equation}
Then by (\ref{eq:1})
\begin{equation}\label{eq:2}
B_n\left(\displaystyle\frac{\theta_3(n\tau)}{\theta_3},\frac{\theta_4}{\theta_3}\right)=0.
\end{equation}
Furthermore, since $Q_m(0,Y)$ is a nonzero constant by Lemma \ref{lem:1}, 
we can apply (\ref{eq:P}) to get
$$c:=Q_m(0,Y)=Q_m(0,1)=\sum_{\mu \geq 0}^{}a_{0,\mu},$$
and hence by (\ref{eq:B}) there exists $A_{n,j}(X)$ for each $j$ with $0\leq j<2\psi(m)$ such that
\begin{eqnarray}
B_n(X,Y)&=&\left(
\displaystyle\sum_{\mu\geq0}a_{0,\mu}
\right)Y^{4\psi(m)}+
\displaystyle\sum_{j=0}^{2\psi(m)-1}A_{n,j}(X^4)Y^{2j}\nonumber \\\nonumber\\
{}&=&cY^{4\psi(m)}+
\displaystyle\sum_{j=0}^{2\psi(m)-1}A_{n,j}(X^4)Y^{2j}\label{eq:0603},
\end{eqnarray}
where $\deg B_n(X,Y)=4\psi(m)$ follows again from (\ref{eq:B}). 
We rewrite $B_n(X,Y)$ by
$$
\begin{array}{ll}
B_n(X,Y)&=\displaystyle\sum_{j\geq 0} C_{n,j}(X^4)Y^{2j}\\
{}&=\displaystyle\sum_{\scriptsize{ \begin{array}{c} j\geq 0 \\ j\equiv0\pmod{2} \end{array}}} C_{n,j}(X^4)Y^{2j}+
\displaystyle\sum_{\scriptsize{ \begin{array}{c} j\geq 1 \\ j\equiv1\pmod{2} \end{array}}} C_{n,j}(X^4)Y^{2j}\\\\
{}&=:D_n(X^4,Y^4)+Y^2E_n(X^4,Y^4) \,.
\end{array}
$$
Define
\begin{eqnarray*}
\tilde{Q}_n(X,Y) &:=& D_n^2(X,Y)-YE_n^2(X,Y) \,,\\
Q_n(X,Y) &:=& \tilde{Q}_n(X,1-Y) \,.
\end{eqnarray*}
Note that 
\begin{eqnarray}
\tilde{Q}_n(X^4,Y^4)&=&D_n^2(X^4,Y^4)-Y^4E_n^2(X^4,Y^4) \nonumber \\
{}&=&B_n(X,Y)B_n(X,iY)
\label{eq:4}.
\end{eqnarray}
Substituting $X=\theta_3(n\tau)/\theta_3$ and $Y=\theta_4/\theta_3$ into this identity, we have by (\ref{eq:2}) 
$$
\tilde{Q}_n\left(\displaystyle\frac{\theta_3^4(n\tau)}{\theta_3^4},\frac{\theta_4^4}{\theta_3^4}\right)=0.  
$$
Since $\theta_4^4/\theta_3^4=1-\theta_2^4/\theta_3^4$, it is clear that
$$
Q_n\left(\displaystyle\frac{\theta_3^4(n\tau)}{\theta_3^4},\frac{\theta_2^4}{\theta_3^4}\right)=0.
$$
Furthermore, by (\ref{eq:4}) together with (\ref{eq:0603}) 
$$
\begin{array}{ll}
\tilde{Q}_n(X^4,Y^4)&=
\left(
cY^{4\psi(m)}+
\displaystyle\sum_{j=0}^{2\psi(m)-1}A_{n,j}(X^4)Y^{2j}\right)
\left(
cY^{4\psi(m)}+
\displaystyle\sum_{j=0}^{2\psi(m)-1}(-1)^jA_{n,j}(X^4)Y^{2j}\right)\\\\
{}&=c^2Y^{8\psi(m)}+
\displaystyle\sum_{j=0}^{2\psi(m)-1}S_{n,j}(X^4)Y^{4j},
\end{array}
$$
where for each $j$ with $0\leq j<2\psi(m)$ we have
$$
\deg \tilde{Q}_n(X^4,Y^4)\,=\, 2\deg B_n(X,Y) \,=\, 8\psi(m)\geq \deg \big( S_{n,j}(X^4)Y^{4j}\,\big) \,=\, 4\deg S_{n,j}(X)+4j,
$$
which implies
$$
\deg S_{n,j}(X)\leq 2\psi(m)-j.
$$
Hence we get  
$$
\tilde{Q}_n(X,Y)=c^2Y^{2\psi(m)}+
\displaystyle\sum_{j=0}^{2\psi(m)-1}S_{n,j}(X)Y^{j},
$$
and so  
$$
Q_n(X,Y)=\tilde{Q}_n(X,1-Y)=
c^2Y^{2\psi(m)}+
\displaystyle\sum_{j=0}^{2\psi(m)-1}R_{n,j}(X)Y^{j},
$$
where $\deg R_{n,j}(X)\leq 2\psi(m)-j$ $(0\leq j<2\psi(m))$. It remains to prove (\ref{eq:leadcoeff2}) in the case of $\alpha =1$. 
From Lemma~\ref{lem:1} we get
\[c \,=\, P_m(0,Y) \,=\, \sum_{\mu \geq 0} a_{0,\mu}Y^{\mu} \,,\]
so that $a_{0,0}=c$ and $a_{0,\mu}=0$ for $\mu \not= 0$. Then, by (\ref{eq:P}) we obtain
\[B_n(0,Y) \,=\, \sum_{\mu \geq 0} a_{0,\mu} {(1-Y^2)}^{2\mu}{(1+Y^2)}^{2(\psi(m) - \mu)} \,=\, c{(1+Y^2)}^{2\psi(m)}\,.\]
By (\ref{eq:4}) we have
\[\tilde{Q}_n(0,Y^4) \,=\, c^2{(1+Y^2)}^{2\psi(m)}{(1-Y^2)}^{2\psi(m)} \,=\, c^2{(1-Y^4)}^{2\psi(m)} \,,\]
and therefore
\[\tilde{Q}_n(0,Y) \,=\, c^2{(1-Y)}^{2\psi(m)} \,,\]
which gives
\[Q_n(0,Y) \,=\, c^2Y^{2\psi(m)} \]
by the above definition of $Q_n(X,Y)$.
Hence the proof of Lemma \ref{lem:1} with $\alpha=1$ is completed.

Next, let the lemma be true for some fixed $\alpha \geq 1$ with $n=2^{\alpha}m$. In the preceding part of the proof we replace $m$ by $n$ and the polynomial $P_m(X,Y)$ 
by the polynomial
\[Q_n(X,Y) \,=\, \sum_{\nu,\mu} b_{\nu,\mu}X^{\nu}Y^{\mu} \]
satisfying (\ref{eq:-1}) to (\ref{eq:leadcoeff2}). In particular, 
\begin{equation}\label{eq:963}
\deg Q_n(X,Y)=2^\alpha\psi(m)
\end{equation}
and 
\begin{equation}
b_{0,\mu} \,=\,\,\left\{ \begin{array}{ll} c^{2^{\alpha}},& \quad \mbox{if} \quad \mu = 2^{\alpha}\psi(m)\,,\\
0, & \quad \mbox{otherwise}\,.
\end{array} \right.
\label{eq:b}
\end{equation}
Replacing
$\tau$ by $2\tau$ we thus obtain instead of (\ref{eq:35})
\[0=Q_n\left(\displaystyle\frac{\theta_3^4(2n\tau)}{\theta_3^4(2\tau)},\frac{\theta_2^4(2\tau)}{\theta_3^4(2\tau)}\right)
=\displaystyle\sum_{\nu,\mu}b_{\nu,\mu}\left(\frac{\theta_3^4(2n\tau)}{\theta_3^4(2\tau)}\right)^{\nu}
\left(\frac{\theta_2^4(2\tau)}{\theta_3^4(2\tau)}\right)^{\mu} \,.\]
By the method described above for the case $\alpha =1$ we obtain again polynomials $B_{2n}(X,Y)$, $\tilde{Q}_{2n}(X,Y)$, and $Q_{2n}(X,Y)$ such that
\begin{eqnarray*}
\tilde{Q}_{2n}(X^4,Y^4) &:=& B_{2n}(X,Y)B_{2n}(X,iY) \,,\\
Q_{2n}(X,Y) &:=& \tilde{Q}_{2n}(X,1-Y) 
\end{eqnarray*}
and, step by step,
\begin{eqnarray*}
0 &=& \tilde{Q}_{2n}\left(\displaystyle\frac{\theta_3^4(2n\tau)}{\theta_3^4},\frac{\theta_4^4}{\theta_3^4}\right) \,,\\
0 &=& Q_{2n}\left(\displaystyle\frac{\theta_3^4(2n\tau)}{\theta_3^4},\frac{\theta_2^4}{\theta_3^4}\right) \,.
\end{eqnarray*}
This proves (\ref{eq:-1}) for $n$ replaced by $2n=2^{\alpha +1}m$. 
Next, we consider (\ref{eq:leadcoeff}). We have by (\ref{eq:963}) and (\ref{eq:b}) instead of (\ref{eq:0603})
$$
\begin{array}{ll}
B_{2n}(X,Y) \,&=\, \displaystyle\sum_{\nu,\mu} 2^{2\nu}b_{\nu,\mu}X^{4\nu}{(1-Y^2)}^{2\mu} {(1+Y^2)}^{2( 2^{\alpha}\psi(m) - \nu - \mu)} \\\\
{}&=
\left(
\displaystyle\sum_{\mu\geq0}b_{0,\mu}
\right)Y^{2^{\alpha+2}\psi(m)}+
\displaystyle\sum_{j=0}^{2^{\alpha+2}\psi(m)-1}A_{2n,j}(X^4)Y^{2j}\nonumber \\\nonumber\\
{}&=c^{2^{\alpha}}Y^{2^{\alpha+2}\psi(m)}+
\displaystyle\sum_{j=0}^{2^{\alpha+2}\psi(m)-1}A_{2n,j}(X^4)Y^{2j},
\end{array}
$$
where $\deg B_{2n}(X,Y)=2^{\alpha+2}\psi(m)$. 
Then, by the same arguments as in the case of $\alpha=1$, we obtain 
(\ref{eq:leadcoeff}) with $\alpha$ replaced by $\alpha +1$. Finally, we show (\ref{eq:leadcoeff2}) for $\alpha +1$. With the above formula for $B_{2n}(X,Y)$ we 
obtain
\begin{eqnarray*}
B_{2n}(0,Y) &=& \sum_{\mu \geq 0} b_{0,\mu} {(1-Y^2)}^{2\mu}{(1+Y^2)}^{2(2^{\alpha}\psi(m)-\mu)} \\
&=& b_{0,2^{\alpha}\psi(m)}{(1-Y^2)}^{2^{\alpha +1}\psi(m)} \\
&=& c^{2^{\alpha}}{(1-Y^2)}^{2^{\alpha +1}\psi(m)} \,,
\end{eqnarray*}
where we used (\ref{eq:b}). Thus, using the same arguments as in the case $\alpha =1$, we obtain
\[\tilde{Q}_{2n}(0,Y^4) \,=\, c^{2^{\alpha +1}}{(1-Y^4)}^{2^{\alpha +1}\psi(m)}\,,\] 
hence  
\[Q_{2n}(0,Y) \,=\, c^{2^{\alpha +1}}Y^{2^{\alpha +1}\psi(m)}\,.\]   
This completes the proof of the lemma. \hfill 
\end{proof}
\begin{example}\label{Ex2}
With the polynomial from Example~\ref{Ex1} we obtain for $n=6$ 
\begin{eqnarray*}
B_6(X,Y) &=& 9Y^{16}+72Y^{14}+(-1008X^4+252)Y^{12}+(30816X^4+504)Y^{10} \\
&& +\,(38880X^8-15120X^4+630)Y^8+(155520X^8-93888X^4+504)Y^6 \\
&& +\,(-559872X^{12}+233280X^8-15120X^4+252)Y^4 \\
&& +\,(-1119744X^{12}+155520X^8+30816X^4+72)Y^2 \\
&& +\,1679616X^{16}-559872X^{12}+38880X^8-1008X^4+9
\end{eqnarray*}
and
\begin{eqnarray*}
Q_6(X,Y) &=& 81Y^8+18144XY^7+(1715904X^2-5344704X)Y^6 \\
&& +\,(88459776X^3+907448832X^2+58392576X)Y^5 \\
&& +\,(2670589440X^4-11804341248X^3+1470721536X^2-180486144X)Y^4 \\
&& +\,(46921752576X^5-92553560064X^4 \\
&& \quad +\,34882265088X^3-4756340736X^2+212336640X)Y^3 \\
&& +(444063596544X^6-148021198848X^5+96423395328X^4 \\
&& \quad -\,23254843392X^3+2378170368X^2-84934656X)Y^2 \\
&& +\,(1880739938304X^7-1044855521280X^6+162533081088X^5-7739670528X^4)Y \\
&& +\,2821109907456X^8-3761479876608X^7+1044855521280X^6-108355387392X^5 \\
&& \quad +\,3869835264X^4 \,.
\end{eqnarray*}
\end{example}

\section{Proof of Theorem~\ref{Thm1}} \label{Sec4} 

Let $n=2^\alpha m$ $(\alpha,m\geq 1)$ with an odd number $m$, 
and let $\tau\in\mathbb{H}$ be as in Theorem \ref{Thm1}. 
Let $F:=\mathbb{Q}(\theta_3(n\tau),\theta_3(\tau))$. 
Then by Lemma \ref{lem:3} 
there exists a 
polynomial $Q_n(X,Y)\in\mathbb{Z}[X,Y]$ satisfying (\ref{eq:-1}) and (\ref{eq:leadcoeff}), 
namely, the number $Y=\theta_2^4$ is zero of a nonzero polynomial 
$g(Y)\in F(Y)$ defined by 
$$
g(Y) \,:=\, Q_n\left(
\frac{\theta_3^4(n\tau)}{\theta_3^4},\frac{Y}{\theta_3^4}
\right)=
c^{2^{\alpha}}\left(
\frac{Y}{\theta_3^4}
\right)^{2^{\alpha}\psi(m)}+
\sum_{j=0}^{2^{\alpha}\psi(m)-1}
R_{n,j}\left(
\frac{\theta_3^4(n\tau)}{\theta_3^4}
\right)
\left(
\frac{Y}{\theta_3^4}
\right)^j,
$$
which implies that the number $\theta_2$ is algebraic over the 
field over $F$. 
Hence, putting $E:=F(\theta_2,\theta_3)$, we get
\[2 \leq \, \mbox{trans.\,deg\,} E/{\Q} \,=\, \mbox{trans.\,deg\,} E/F + \mbox{trans.\,deg\,} F/{\Q} \,=\, \mbox{trans.\,deg\,} F/{\Q} \,,\]
where for the inequality on the left-hand side 
we used the algebraic independence of the numbers 
$\theta_2$ and $\theta_3$ 
(cf. \cite[Theorem~4]{Bertrand}, \cite[Lemma~4]{Elsner2}). 
On the other hand, 
$\mbox{trans.\,deg\,} F/{\Q} \leq 2$ is trivial. Therefore we obtain
\[\mbox{trans.\,deg\,} F/{\Q} \,=\, 2 \,,\]
which shows that the numbers $\theta_3(n\tau)$ and $\theta_3(\tau)$ are 
algebraically independent over $\mathbb{Q}$. 
\hfill \qed

\section{Proof of Theorem~\ref{Thm3}}\label{Sec5} 

In this section, let $\tau\in\mathbb{H}$ be a complex number 
such that \(q=e^{\pi i \tau}\) is an algebraic number.

\begin{lemma}\label{lem:06061}
Let $n\geq1$ be an integer. 
Then the number $\theta_3(n\tau)$ is algebraic over the field 
$F:=\mathbb{Q}(\theta_3,\theta_4)$. 
\end{lemma}
\begin{proof} 
The assertion is trivial for $n=1$. Let $n\geq2$. 
By \cite[Lemma 3.1]{Elsner3}, Lemma \ref{lem:1}, and Lemma \ref{lem:3}, there exists a nonzero 
integer polynomial $R_n(X,Y)$ such that 
$$
R_n\left(
\frac{\theta_3(n\tau)}{\theta_3},\frac{\theta_4}{\theta_3}
\right)=0.
$$ 
Define 
$$
T_n(X):=R_n\left(
\frac{X}{\theta_3},\frac{\theta_4}{\theta_3}\right)\in F[X].
$$
It is clear that $T_n(\theta_3(n\tau))=0$. Furthermore, the 
polynomial $T_n(X)$ 
is nonzero, since the numbers $\theta_3$ and $\theta_4$ are algebraically independent over $\mathbb{Q}$ 
(cf. \cite[Theorem~4]{Bertrand}, \cite[Lemma~4]{Elsner2}). 
From this fact the statement of the lemma follows.
\end{proof}
{\em Proof of Theorem~\ref{Thm3}.\/} 
Let $F:=\mathbb{Q}(\theta_3,\theta_4)$ and 
$$
E:=\mathbb{Q}\big( \theta_3(\ell \tau), \theta_3(m\tau),\theta_3(n\tau)\big).
$$
Then by Lemma \ref{lem:06061} we see that the composite 
$EF$ of the fields $E$ and $F$ is an algebraic extension of $F$, and hence
$$
\begin{array}{ll}
{\rm trans.}\deg E/\mathbb{Q}&\leq
{\rm trans.}\deg EF/\mathbb{Q}\\
{}&={\rm trans.}\deg EF/F+{\rm trans.}\deg F/\mathbb{Q}\\
{}&={\rm trans.}\deg F/\mathbb{Q}\\
{}&=2,
\end{array}
$$
which is our desire. \hfill $\Box$

\section{Comments}\label{Sec6}  
{\em (i)\/} 
\,Let $\tau\in\mathbb{H}$. 
From the polynomials $P_2$ and $P_3$ stated in Section~\ref{Sec1} we obtain the polynomial
\[P(X,Y,Z) \,:=\, 27X^8 - 18X^4Y^4 - 64X^2Y^4Z^2 + 64X^2Y^2Z^4 - 8X^2Z^6 - Z^8 \,,\]
which vanishes for
\begin{eqnarray*}
X &=& \theta_3(3\tau) \,,\\
Y &=& \theta_3(2\tau) \,,\\
Z &=& \theta_3(\tau) \,.
\end{eqnarray*}
Hence, these numbers are homogeneously algebraically dependend over ${\Q}$. Note that this is true for {\em all\/} complex numbers 
$\tau\in\mathbb{H}$, 
not only for algebraic numbers $q=e^{\pi i\tau}$ as stated in Theorem~\ref{Thm3}.
\\               
{\em (ii)\/} \,An important property of the polynomials $Q_n(X,Y)$ in Lemma~\ref{lem:3} we needed is given by (\ref{eq:leadcoeff2}), namely
\[Q_n(0,Y) \,=\, c^{2^{\alpha}}Y^{2^{\alpha}\psi(m)} \,,\] 
where $c=P_m(0,Y)$. We state some results on the values of $c$.  
By (\ref{eq:comment}) and $P_m(0,Y)=R_{\psi(m)}$ it turns out that
\[P_m(0,Y) \,=\, {\Big( \,\prod_{\scriptsize{\begin{array}{c} d|m \\ d\geq 1 \end{array}}}\,d^{\omega(d,m/d)} 
\Big)}^2 \,,\]
where 
\[\omega (a,b) \,:=\, \sum_{\scriptsize{\begin{array}{c} (a,b,k)=1 \\ 0\leq k<b \end{array}}} 1 \]
for arbitrary positive integers $a$ and $b$.
\begin{lemma}
Let $a$ and $b$ be positive integers, and let $\varphi(n)$ denote Euler's totient. By $(a,b)$ we denote the greatest common divisor of $a$ and $b$. Then,
\[\omega (a,b) \,=\, \frac{b}{(a,b)}\varphi \big( (a,b) \big) \,.\]
\label{Lem0}
\end{lemma} 
{\em Proof.\/} \,We obtain for positive integers $\l$ and $m$
\begin{eqnarray*}
\sum_{\scriptsize{\begin{array}{c} (m,k)=1 \\ 0\leq k<\l m \end{array}}} 1 &=& \sum_{\nu =1}^{\l} \,\sum_{\scriptsize{\begin{array}{c} (m,k)=1 \\ (\nu -1)m\leq 
k<\nu m \end{array}}} 1 \,=\, \sum_{\nu =1}^{\l} \,\sum_{\scriptsize{\begin{array}{c} \big(m,(\nu -1)m+k\,\big)=1 \\ 0\leq k<m \end{array}}} 1 \\
&=& \sum_{\nu =1}^{\l} \,\sum_{\scriptsize{\begin{array}{c} (m,k)=1 \\ 0\leq k<m \end{array}}} 1 \,=\, \sum_{\nu =1}^{\l} \varphi(m) \\
&=& \l \varphi (m) \,.
\end{eqnarray*}
Setting $m:=(a,b)$ and $\l :=b/(a,b)$, we find
\[\omega (a,b) \,=\, \sum_{\scriptsize{\begin{array}{c} \big((a,b),k\big)=1 \\ 0\leq k < \cfrac{b}{(a,b)}\,(a,b) \end{array}}} 1 \,=\, 
\frac{b}{(a,b)}\varphi \big( (a,b) \big) \,,\]
as desired. \hfill \qed \\

\begin{example}\label{Ex3}
\begin{itemize}
\item[1.)] Let $p\geq 3$ be a prime and $\alpha \geq 1$. Then, 
\[P_{p^{\alpha}}(0,Y) \,=\, p^{2(p^{\alpha}-1)/(p-1)} \,.\]
\item[2.)] By straightforward computations one obtains
\[P_{15}(0,Y) \,=\, 3^{12}\cdot 5^8 \,=\, 207\,594\,140\,625 \,.\]
\item[3.)] Let $m=p_1^{\alpha_1}\cdots p_r^{\alpha_r}$ with $\alpha_{\nu}\geq 1$ and $d=p_1^{\beta_1}\cdots p_r^{\beta_r}$ with $0\leq \beta_{\nu}\leq \alpha_{\nu}$
$(\nu =1,\dots ,r)$ be a divisor of $m$. Set
\[\gamma_{\nu} \,:=\, \min \big\{ \beta_{\nu},\,\alpha_{\nu} - \beta_{\nu} \big\} \qquad (\nu =1,\dots ,r) \,.\]
Note that every $\gamma_{\nu}$ depends on a divisor $d$ of $m$. Then, we have 
\[\omega \Big( \,d,\frac{m}{d}\,\Big) \,=\, p_1^{\alpha_1 - \beta_1} \cdots p_r^{\alpha_r - \beta_r} \cdot \prod_{\scriptsize{\begin{array}{c}
\nu =1 \\ \gamma_{\nu} >0 \end{array}}}^r \Big( \,1 - \frac{1}{p_{\nu}}\,\Big) \,.\]
\end{itemize} 
\end{example}

{\bf Acknowledgments.} 
The main part of this work was carried out during the second author's visit in 
the FHDW, Hannover, in June of 2016. 
He would like to express his sincere gratitude to the staff of 
the FHDW for their warm hospitality.  
The second author was also supported by Japan Society for the Promotion of Science, Grant-in-Aid for Young Scientists (B), 15K17504.


\end{document}